\documentclass[10pt,a4paper]{article}

\usepackage{amsmath,amssymb,latexsym,color,dsfont,mleftright}
 \usepackage{authblk,mathtools, nccmath,subcaption}
 \usepackage{graphicx,float}
 \usepackage[square,numbers]{natbib}
 \usepackage{natbib}
 \usepackage{tikz-cd}
  \usepackage{bbm}
 \usepackage[cal=boondox]{mathalfa}
  \usepackage{cleveref}
\usepackage{ntheorem} 
\usepackage[colorinlistoftodos]{todonotes}
\usepackage{CJKutf8}
\usepackage{float}
\usepackage[export]{adjustbox}
\usepackage[T1]{fontenc}
\usepackage{esint}
\usepackage{multirow}
\usepackage{booktabs}
\usepackage{changepage} 
\usepackage{array}

\newtheorem{theorem}{Theorem}

\newtheorem{lemma}{Lemma}
\newtheorem{proposition}{Proposition}

\newtheorem{definition}{Definition}
\newtheorem{coro}{Corollary}

\newtheorem{problem}{Problem}
\newtheorem{remark}{Remark}

\newcommand{\lqqd}{\hfill{$\Box$}\bigskip}
\newcommand{\CC}{\mathbb{C}}

\newcommand{\RR}{\mathbb{R}}
\newcommand{\ZZ}{\mathbb{Z}}

\newcommand{\dsty}{\displaystyle}
\newcommand{\proof}{\bigskip\noindent {\sc Proof.  }}

\def\bbbuildrel#1_#2{\mathrel{
\mathop{\kern 0pt#1}\limits_{#2}}}

\begin{document}

\title{Liouvillian and Analytic Integrability of a Generalized Gause System }

\author[1]{Jorge A. Borrego-Morell}

\affil[1]{Universidade Federal do Rio de Janeiro, Campus Duque de Caxias, Rio de Janeiro, Brazil\newline
{\small \texttt{Email: jborrego@caxias.ufrj.br; ORCID: 0000-0001-8871-2739}}}
\maketitle

\begin{abstract}

 In this work, we identify the regions of the parameter space in which  a predator–prey system—derived from the classical Gause model with a generalized Holling response function and logistic prey growth in the absence of predators—fails to be Liouvillian integrable. Although the model parameters have biological meaning only when restricted to appropriate real domains, our analysis is carried out in the complex setting, which provides a unified algebraic framework; the resulting nonintegrability conditions remain valid in the biologically relevant regime. As a consequence, we establish the nonintegrability of an Abel differential equation of the second kind with polynomial coefficients obtained from the system. Finally, we analyze the existence of a local analytic first integral in neighborhoods of the  equilibrium points.
 
\end{abstract}

\noindent {\it MSC2020 classes:}   Primary 34A34, 34CA05.\\


\section{Introduction}

The aim of this paper is to study the Liouvillian integrability and analyticity
of solutions of a generalized predator–prey Gause system \cite[Ch.~4]{Fri80}
\begin{equation}\label{Ga1}
  \begin{cases}
    \dot{x}=  x g(x) - y p(x), \\
    \dot{y} =  y \left(-\gamma + c\, p(x)\right),
  \end{cases}
\end{equation}
where $g$ denotes the specific growth rate of the prey in the absence of
predators, and $p$ is the functional response of the predator with respect to
that prey. We assume that the prey follows a logistic growth law,
$g(x)=r\left(1-\frac{x}{k}\right)$ \cite[Ch.~1]{Fri80}, and that the predation is
described by a  generalized Holling functional response
\[
    p(x)=\frac{\alpha x^m}{\beta + x^m}, \qquad m\in\mathbb{Z}_{>0},
\]
in  particular, when 
$m=1$, system \eqref{Ga1} reduces to the classical Rosenzweig–MacArthur model \cite[p.~95]{Tu13}.
 
From a biological perspective, the functional response defined  above requires $\beta>0$, because  $\beta$ plays the role of a half-saturation constant.
Nonetheless, for mathematical completeness, we also consider the case $\beta=0$.  In this regime the functional response becomes constant and the system no longer represents a biologically meaningful predator–prey model; consequently,  the corresponding results should  be interpreted purely in a mathematical sense.

 Although the parameters of the Gause system have biological meaning only when they belong to appropriate subsets of the real field 
$\RR$, in our analysis we allow them to vary in the complex domain 
$\CC$, with $k\neq 0$ and $ m\in\mathbb{Z}_{>0}$. This extension does not affect the conclusions for the real parameter regime; rather, it provides a unified and algebraically complete setting in which integrability can be fully characterized. Our analysis relies on Darboux theory and its extensions—an algebraic framework valid over any field of characteristic zero. Therefore, the integrability conditions obtained in the complex domain specialize consistently to the biologically relevant real case. Consequently, the  integrability or nonintegrability established over 
$\CC$ automatically determines the corresponding behavior in the real, ecologically meaningful parameter regime.

The vector field associated with system \eqref{Ga1}, is given by
$$
 L_{v}=-\frac{\frac{r}{k}\Pi_{m+2}(x)+\alpha  x^m y  }{\beta+x^m}\frac{\partial}{\partial x}+\frac{(- \gamma\beta+ (\alpha c-\gamma)x^m )y}{  \beta+x^m }\frac{\partial}{\partial y},
$$
where $\Pi_{m+2}(x)=x\left(x-k\right)(x^m+\beta)$ is a polynomial of degree $m+2$.

We distinguish two parameter regimes. If $\beta\neq 0$, both components of $L_v$ share the common factor $(\beta+x^m)^{-1}$. Multiplying by $\beta+x^m$ yields the polynomial vector field $L_{v_1}$ defined  below. Given that  multiplication of a vector field by a nonvanishing scalar function does not alter its phase curves but only reparametrizes time along trajectories, the study of first integrals of $L_v$ is equivalent to that of $L_{v_1}$.  If $\beta=0$, the functional response reduces to the constant 
$p(x)=\alpha$, and system \eqref{Ga1} simplifies directly to a polynomial 
system; the corresponding vector field is denoted by  $L_{v_2}$ below. Hence, the study of first integrals can  be reduced to the following  polynomial 
vector fields
\begin{equation}
\begin{aligned}\label{vf}
L_{v_1}&= - \left(\frac{r}{k}\Pi_{m+2}(x) +\alpha  x^m y  \right)  
\frac{\partial}{\partial x}
+\left(- \gamma\beta+ (\alpha c-\gamma)x^m \right)y 
\frac{\partial}{\partial y}, 
\quad \beta \neq 0, \\[4pt]
L_{v_2}&=  -\left(\frac{r}{k}x(x-k)+\alpha y\right) 
\frac{\partial}{\partial x}
+y(-\gamma+\alpha c) 
\frac{\partial}{\partial y}, 
\quad \beta = 0.
\end{aligned}
\end{equation}

The predator–prey system \eqref{Ga1} generalizes the classical Gause predator–prey model introduced in \cite{G34,G36}. Named after the Russian ecologist Georgii Gause, this framework incorporates nonlinear functional responses and density-dependent prey growth, yielding  more complex and biologically plausible dynamics. It represents a significant  advancement over the planar Lotka–Volterra system \cite{Lo20,Vo31},
 \begin{eqnarray*}
   \begin{cases}
\dot{x}=A x -B xy, \\
\dot{y}=C xy - D y,
\end{cases}  
\end{eqnarray*}
by including mechanistic details of predation and resource limitation.

 The integrability of planar Lotka-Volterra systems,   
 \begin{eqnarray*}
   \begin{cases}
\dot{x}=x(Ax  - B y+c_1), \\
\dot{y}=y(C x - Dy +c_2),
\end{cases}  
\end{eqnarray*}
 has been studied in \cite{CaGiLl07,CaGiLl03} for both real and complex  parameter values using the Darboux method. For higher-dimensional cases, see \cite{O97,O99,O01} and the references therein. Relatively few planar polynomial systems are known to admit Liouvillian or analytic first integrals (see \cite{FVW18,LlV012} and references therein). Along these lines, recent works have identified new cases of  Liouvillian integrability  among nonlinear oscillators, including general Rayleigh–Duffing-type systems and related biochemical models (see \cite{Gin22,GinV19,GinV25}).

The study of first integrals depending upon parameters is a classical problem in the qualitative theory of planar differential equations, which has been addressed using  specific methods (see, e.g., \cite{CaGiLl03}).

The classical Darboux theory on the  integrability of polynomial vector fields, established  by Jean–Gaston Darboux \cite{Darb,Darb78},  concerns  with the existence of rational first integrals and remains a powerful tool in this context. This method relies on finding a sufficient number of invariant algebraic curves, now known as  Darboux polynomials. Poincar\'e \cite{Poin91,Poin97}, was the first to recognize the difficulty of determining  in general whether a polynomial system admits such invariant algebraic solutions. One of his celebrated open problems asks whether, for a given  polynomial system of degree $m \ge 2$, there exists an upper bound $N(m)$ for the degree of its algebraic solutions. In general, such a bound cannot depend solely  on $m$ and may vary with the particular system under consideration (see \cite[Ch.4]{Zh17} and also \cite{Gin07}, in which  examples of planar systems illustrating this dependence are provided).

 The problem of finding Darboux polynomials has been studied extensively. Under certain assumptions, efficient computational algorithms exist (see \cite{BoCh16,Che11}), though they often require a degree bound or an a priori assumption on the degree. A vast body of work also addresses the computation  of rational first integrals and bounded families of Darboux polynomials (see \cite[Sect.~2.3]{BoCh16}, also  \cite{Che11,LlVZ16,Sch93} for
some historical remarks). Recently, the authors of  \cite{ChCo20} developed a semialgorithm for computing rational, Darbouxian, Liouvillian, or Riccati first integrals without explicitly calculating Darboux polynomials. Nonetheless, proving the existence of Darboux polynomials, or bounding their degree, often remains a formidable  challenge. As shown below, the polynomial system \eqref{vf} admits Darboux polynomials of unbounded degree, rendering  algorithmic  approaches inapplicable. While traditional algorithms often require an a priori degree bound, the Puiseux series method \cite{D22,DGV22} provides a systematic framework to determine integrability without this constraint. This approach has led to significant classification results for several regimes of polynomial Li\'enard systems, establishing necessary and sufficient conditions for Liouvillian integrability in cases in which  the degrees of the defining polynomials satisfy specific relations \cite{D23}.

Our analysis of Liouvillian integrability is based  on the results of M.F. Singer \cite{Sing92} and C. Christopher \cite{Chr99}, who showed that for planar polynomial systems the existence of a Liouvillian first integral is equivalent to the existence of an integrating factor of extended Darboux type—that is, one constructed from Darboux polynomials, exponential factors, and rational combinations of these objects. Consequently, we deduce results on the Liouvillian integrability of an  Abel differential equation obtained from  system \eqref{Ga1}. The study of Abel equations has a long history, dating back to Abel, Liouville, and Appell in the nineteenth century. Despite their central role in the theory of ODEs, no general method exists for obtaining closed-form solutions beyond classical integrable subclasses (see Kamke \cite{Ka59}, Polyanin–Zaitsev \cite{Po03}). Even recent approaches—using symmetries, Darboux theory, or equivalence classifications (see \cite{AJLS24,Bo00,CGR09,Che00,Che03,Gi10,Op22})—have succeeded only  in  identifying  isolated particular solvable cases. This motivates rigorous proofs of nonintegrability, which clarify the limitations of analytic methods and establish precise boundaries between solvable and nonsolvable cases.

We also study the existence of local analytic first integrals in the  neighborhoods of equilibrium points. In this context,  we rely on the  classical Poincar\'e’s resonance criterion, which is discussed in the next section. This topic  is closely connected to the center problem: determining the conditions under which a singularity of a real  planar vector field is a center—that is, a point in which  all nearby trajectories are closed orbits (see \cite{Ch03,Gin07,Gin13}). It is a classical result—originating with Poincar\'e \cite{Poi81,Poi51} and rigorously established by Lyapunov \cite{Lya07}—that a singular point of a real planar analytic system whose linear part at the point  has purely imaginary eigenvalues (the nondegenerate case) is a center if and only if it admits a non-constant local analytic first integral. While Poincar\'e’s analysis of the return map yielded a formal first integral, the crucial proof of series convergence in the analytic category was provided by Lyapunov \cite{Lya07}.

 This fundamental result illustrates the importance of analytic first integrals in the study of local and global dynamics. However, constructing such integrals, or even proving their existence, remains a challenging problem. In this manuscript we employ the Poincar\'e approach as the main framework for our analysis of local integrability; while a  complete solution of the center problem is certainly an interesting direction, it is not addressed in detail  here. For broader perspectives on the center problem, including alternative methods and classifications, we refer the reader to the monographs  \cite{ChLi24,Po21}.

 Finally, the manuscript is organized as follows. Section \ref{S2} reviews the notation and basic definitions of  Darboux integrability. Section \ref{S3} presents our results concerning  Liouvillian integrability, and  Section \ref{S4} investigates  analytic integrability.

\section{Basic concepts, notations,  and statement of the results}\label{S2}

Let $\CC[x,y]$ denote the ring of polynomials with complex coefficients in the variables $x$ and $y$. For a subset $I$ of the set of roots of $\dsty \Pi_{m+2}$, we denote by $|I|$ the cardinality of $I$.
%

Given a system of differential equations
\begin{eqnarray} \label{S1}
\begin{cases}
\dot{x}= p(x,y), \\
\dot{y} = q(x,y),
\end{cases}
\end{eqnarray}
in which  $p$ and $q$ are functions of class $C^k(U)$, $k\geq 1$, and $U\subset \CC^2$ is open, the vector field $L$ associated with system \eqref{S1} is defined by
\begin{eqnarray}\label{L}
L=p\frac{\partial }{\partial x}+q\frac{\partial }{\partial y},
\end{eqnarray}
if $p$ and $q$ are polynomials, the degree of the polynomial system is defined as $M=\max\{\deg[p],\deg[q]\}$.

Let $U$ be an open subset of $\CC$ and let $I : U\rightarrow \CC$ be an analytic function which is not identically zero on $U$. The function $I$ is an  integrating factor   for system \eqref{S1} if 
\begin{eqnarray}\label{LH}
L[I] = -I\mathrm{div}(p,q),
\end{eqnarray}
where $\mathrm{div}$ denotes the divergence operator.

Recall (cf. \cite[Ch.~8]{DFL}, \cite[Ch.~1]{Zh17}) that a polynomial system \eqref{S1} is said to be integrable on an open subset $U \subset \CC^2$ if there exists a nonconstant analytic function $H : U \to \CC$, called a first integral of the system, such that $H(x(t),y(t)) = \mathrm{const.}$ for every trajectory $(x(t),y(t)) \subset U$. Equivalently, $H$ is a  first integral of \eqref{S1} if and only if
\begin{eqnarray}\label{fI}
L[H] \equiv 0.
\end{eqnarray}

An analytic first integral is a first integral which is an analytic function. As a particular case, a first integral which is a polynomial is called a polynomial first integral. Roughly speaking, a Liouvillian first integral is a first integral that can be expressed as a Liouvillian function, i.e., obtained through quadratures of elementary functions (see \cite[Sect.~3.3]{Zh17} for the formal definition).
A function $H$ is a Darboux first integral of system \eqref{S1} if it is a first integral and is of Darboux type (also called a Darbouxian function), namely,
$$H(x,y)=e^{\frac{f(x,y)}{g(x,y)}}\prod_{i=1}^sf_i^{\lambda_i}(x,y), $$
where $f,g\in \CC[x,y],  \lambda_i\in \CC$, and $f_i\in \CC[x,y], \deg[f_i]\geq 1, i=1,\ldots ,s$; $f_i$ irreducible over $\CC$.

For a detailed exposition of the Darboux and Liouvillian theory of  integrability we refer the reader to \cite{DFL}. See also \cite{PZ13,Sing92,Zh17} for treatments of Liouvillian integrability in the framework of differential algebra.

 \begin{definition}[\cite{Darb}]\label{I}
An affine algebraic curve $f(x,y)=0$, with $f\in \CC[x,y]$ and $\deg[f]\geq 1$, is said to be invariant for a polynomial differential system \eqref{S1} (equivalently, for the differential equation $qdx - pdy = 0$),  if
\begin{eqnarray}\label{DarbP}
L[f] = \mathfrak{c}f,
\end{eqnarray}
where $\mathfrak{c} \in \CC[x,y]$ and $L$ is the associated vector field \eqref{L}. In this case, $\mathfrak{c}$ is called the cofactor of $f$. Such a polynomial $f$ is usually referred to as a Darboux polynomial of system \eqref{S1}. In particular, a polynomial first integral is a Darboux polynomial with eigenvalue (or cofactor) equal to $0$. For completeness, one may also regard nonzero constant polynomials as trivial Darboux polynomials, with identically zero cofactor, although they are usually excluded by some authors  from consideration.
\end{definition}
\begin{definition}[\cite{Ch94}; see also \cite{ChCoLl99,ChCoLl00}]\label{E}
An expression of the form $e^{\tfrac{f}{g}}$, with $f,g \in \CC[x,y]$ relatively prime, is called an exponential factor for system \eqref{S1} (with associated vector field \eqref{L}) if
\begin{eqnarray}\label{ExpFac}
L\left[e^{\tfrac{f}{g}}\right] = K e^{\tfrac{f}{g}},
\end{eqnarray}
where $K \in \CC[x,y]$ with $\deg[K] \leq M-1$. The polynomial $K$ is called the cofactor of the exponential factor $e^{\tfrac{f}{g}}$.
\end{definition}
The following proposition is fundamental in the proof of our main result (Theorem \ref{Liov}); see \cite[Prop.~8.6, p.~218]{DFL} or \cite[Prop.3.4 p.92]{Zh17}.
\begin{proposition}\label{Link}
If $F = e^{\tfrac{f}{g}}$ is an exponential factor for the polynomial system \eqref{S1}, then $g$ is a Darboux polynomial, and $f$ satisfies
\begin{eqnarray}\label{def[f]}
L[f] = f K_g + g K_F,
\end{eqnarray}
where $K_g$ and $K_F$ are the cofactors of $g$ and $F$, respectively.
\end{proposition}

 Define the sets $A=\{(\alpha, r, c,k,\gamma,\beta)\in \CC^6: r\neq 0, \alpha\neq 0 \} $ and $B=\{(\alpha, r, c,k,\gamma,\beta)\in \CC^6:  \alpha  c- \gamma=0, \gamma\beta=0\}$. Our main theorem can be stated as follows.

\begin{theorem}\label{Liov}
Assume that $(\alpha,r,c,k,\gamma,\beta)\in A\setminus B$ and $m\geq 1$.
Then 
\begin{itemize}
\item[a)] for $L_{v_1}$, the irreducible Darboux polynomials are $x$ and $y$ with respective cofactors $ -\frac{ r}{k}(x-k)(x^m+\beta)
-\alpha x^{m-1}y$ and  $ 
-\gamma\beta+(\alpha c-\gamma)x^m$ ;
\item[b)] for $L_{v_2}$, $y$ is the only irreducible Darboux polynomial with cofactor $\alpha c-\gamma$.
\end{itemize}
\end{theorem}

Because every  Darboux polynomial  is a product of powers of the  irreducible ones, it follows from Theorem \ref{Liov}  that  every   Darboux polynomial of $L_{v_1}$ or $L_{v_2}$ can be expressed  respectively as 
\begin{equation}\label{f(x,y)}
\begin{aligned}
g_1(x,y)=&  x^{n_1}y^{n_2},\\
 \mathcal{K}_{g_1}(x,y)=& -\frac{n_1r}{k}(x-k)(x^m+\beta)-n_2(\gamma\beta-(\alpha c-\gamma)x^m)-n_1\alpha x^{m-1}y, \quad n_1,n_2\in \ZZ_{\geq 0},  \\
g_2(x,y)= & y^{n_2},\\
\mathcal{K}_{g_2}(x,y)= & n_2(\alpha c-\gamma), \quad   n_2\in   \ZZ_{\geq 0}.
 \end{aligned}
\end{equation}

 \begin{theorem}\label{Liov2}
Assume that $(\alpha,r,c,k,\gamma,\beta)\in A\setminus B$ and $m\geq 1$. Then   $L_{v_i},i=1,2$,    does not have Darbouxian integrating  factors. Thus, system \eqref{Ga1}  is not Liouvillian  integrable.  Moreover, in the following cases when  $(\alpha,r,c,k,\gamma,\beta)\notin A$ or $(\alpha,r,c,k,\gamma,\beta)\in B$, the associated vector field for  system \eqref{Ga1} reduces to the   following integrable cases with their respective first integrals: 
\begin{itemize}
\item[a)] for $\alpha=0$,  
\begin{eqnarray*}
L_{v}&=&  - \frac{r}{k}\Pi_{m+2}(x)   \frac{\partial}{\partial x}-\left(\gamma\beta+\gamma x^m \right)y \frac{\partial}{\partial y},\\
H(x, y) &=& y \exp \left( -\frac{\gamma k}{r} \int \frac{ dx}{x(x-k)}  \right).
\end{eqnarray*}

\item[b)] for    $r= 0$, 
\begin{eqnarray*}
L_{v}&=&  -  \alpha  x^m   \frac{\partial}{\partial x}+\left(- \gamma\beta+ (\alpha c-\gamma)x^m \right) \frac{\partial}{\partial y}, \\
H(x, y) &=&  y -\int  \left(\frac{\gamma \beta}{ \alpha x^{m} } - \frac{\gamma - \alpha c}{\alpha}\right) dx,  
\end{eqnarray*}

\item[c)] for  $ \alpha  c- \gamma=0$ and $\gamma\beta=0$, 
\begin{eqnarray*}
L_{v}&=&  - \frac{r}{k}\Pi_{m+2}(x)   \frac{\partial}{\partial x},\\
H(x,y)&=&p(y),
\end{eqnarray*}
where $p$ is any polynomial.
\end{itemize}

\end{theorem}

Recall that an Abel-type differential equation of the second kind is an ODE of the form 
\[
\bigl(b_0(x)+b_1(x)\,y \bigr)\,\frac{dy}{dx} \;=\; a_3(x)\,y^3 + a_2(x)\,y^2 + a_1(x)\,y + a_0(x), 
\]
 as a consequence of Theorem \ref{Liov2} we obtain  the Liouvillian integrability of an Abel differential equation obtained from the system.

\begin{coro}\label{coro}\label{Abel}
Under the same hypothesis of Theorem \ref{Liov2},  the second kind Abel differential equation 
\begin{eqnarray}\label{AbDE}
\left(\frac{r}{k}\Pi_{m+2}(x)+\alpha x^m y\right)\frac{dy}{dx}= (\gamma\beta-(\alpha c-\gamma)x^m)y,
\end{eqnarray}
does not possess   Liouvillian first  integrals. 
\end{coro}

The study of analytic first integrals in a neighborhood of an equilibrium point   is based on the following theorem, recall that a point $(x_0, y_0)$ is called a singular or equilibrium point of \eqref{S1} if $p(x_0, y_0) = q(x_0, y_0) = 0$.

\begin{theorem}[Poincar\'e \cite{Poi81}; \cite{Fu96}] \label{Thp}
Assume that the eigenvalues  $\lambda_1$ and $\lambda_2$ of the linear part at some singular point $(x^*,y^*)$ of a polynomial  vector field $L$  are nonzero and that they do not satisfy any resonance condition of the form
$$\lambda_1k_1+\lambda_2k_2=0, \quad \mbox{for} \quad k_1,k_2\in \ZZ^+ \quad \mbox{with}\quad  k_1+k_2>0.$$
Then the vector field $L$ has no analytic first integrals defined in a neighborhood of $(x^*,y^*)$.
\end{theorem}

\begin{theorem}\label{AI}
The following assertions hold:
\begin{itemize}
\item[a)]
The  isolated  equilibrium points    of $L_{v_1}$  with nonzero eigenvalues in its linear part  are given by  
  \begin{equation}
  \begin{aligned}\label{frakp}
  \mathfrak{p}_{1,j} &= \left(\sqrt[m]{\frac{\beta  \gamma }{\alpha  c-\gamma }}, -\frac{r(\alpha c-\gamma)\Pi_{m+2}\left(\sqrt[m]{\frac{\beta  \gamma }{\alpha  c-\gamma }}\right)}{k\alpha \gamma \beta }   \right), \quad \alpha\gamma r c\left(\sqrt[m]{\frac{\beta  \gamma }{\alpha  c-\gamma }}-k\right)\neq 0, \\
  \mathfrak{p}_2 &=(k,0),  \quad  r\Pi^{\prime}_{m+2}(k)\left(\gamma \beta - (\alpha c - \gamma)k^m\right)  \neq 0,  \\
  \mathfrak{p}_3 &=(0,0), \quad   r\gamma \neq 0,   \\
   \mathfrak{p}_{4,j} &=(\sqrt[m]{-\beta},0), \quad r \Pi^{\prime}_{m+2}(\sqrt[m]{-\beta})  \alpha \beta c  \neq 0.
   \end{aligned}
  \end{equation}
For  $L_{v_2}$,
 \begin{equation}
  \begin{aligned}\label{frakq}
  \mathfrak{q}_1 &=(k,0), \quad r(\alpha c-\gamma)\neq 0,\\
  \mathfrak{q}_2 &=(0,0), \quad  r(\alpha c-\gamma)\neq 0.
   \end{aligned}
\end{equation} 

 Here the index $j$ in $\mathfrak{p}_{1,j}$ and  $\mathfrak{p}_{4,j}$ refers to the chosen branch  in  the $m$-root. 

\item[b)]  The nonresonance requirement  in Poincar\'e Theorem \ref{Thp},  \(\dsty \frac{\lambda_1}{\lambda_2}\notin\mathbb Q_{<0}, \lambda_1,\lambda_2\neq 0\)   for the Jacobian matrix at an equilibrium point $p$  is equivalent to the condition,
$$\frac{4D}{T^2}\notin\mathbb{R}\ \text{ or }\ \frac{4D}{T^2} >0\ \text{ or }\ \sqrt{1-\frac{4D}{T^2}}\notin\mathbb{Q}\   \text{ and  }\  D,T\neq 0,$$
where $D=\det(J(p))$ and $T=\operatorname{tr}(J(p))$.  In such a case, no analytic first integral exists  in a neighborhood of the given isolated equilibrium   point   when the parameters satisfy the   condition:

\begin{table}[H]
\centering
\begin{tabular}{|c|c|}
\hline
 $\mathfrak{p}_{1,j}$  &  $\dsty T=-\frac{r}{k}\Pi'_{m+2}(x_*)+ \frac{mr\alpha\beta c(x_*-k)}{k(\alpha c-\gamma)},
D=-\frac{m\alpha\beta^2\gamma r c(x_*-k)}{k(\alpha c-\gamma)}, x_*=\sqrt[m]{\frac{\beta  \gamma }{\alpha  c-\gamma }}$     \\ \hline
$\mathfrak{p}_2$& $\dsty T=-\frac{r}{k}\Pi'_{m+2}(k)-\gamma\beta+(\alpha c-\gamma)k^m,
D=\frac{r}{k}\Pi'_{m+2}(k)\bigl(\gamma\beta-(\alpha c-\gamma)k^m\bigr)$
 \\ \hline
  $\mathfrak{p}_3$ & $\dsty T=r\beta-\gamma\beta, D=-r\gamma\beta^2$    \\ \hline
 $\mathfrak{p}_{4,j}$ &  $T=-\frac{r}{k}\Pi'_{m+2}(x_*)-\alpha\beta c, 
D=\frac{r}{k}\Pi'_{m+2}(x_*)\alpha\beta c, x_*=\sqrt[m]{-\beta}$      \\ \hline
\end{tabular}
\caption{Trace and determinant for the Jacobian at the equilibrium point $\mathfrak{p}_i$ ($\beta\neq 0$)}
\label{betan0}
\end{table}

\begin{table}[H]
\centering
\begin{tabular}{|c|c|}
\hline
$\mathfrak{q}_{1}$ & $\dsty T=-r-\gamma+\alpha c, D=-r(-\gamma+\alpha c) $   \\ \hline
$\mathfrak{q}_2$ & $\dsty  T=r-\gamma+\alpha c, D=r(-\gamma+\alpha c) $ \\ \hline
\end{tabular}
\caption{Trace and determinant for the Jacobian at the equilibrium point $\mathfrak{q}_i$ ($\beta=0$) }
\label{bet0}
\end{table}

\end{itemize}

\end{theorem}

An isolated equilibrium point $p$ of a differential system of class $C^{k}(U)$ in the plane $\RR^2$ is a center if there exists a neighborhood $V\subset U$ of $p$ such that $U\setminus \{p\}$ is filled with periodic orbits. A difficult classical
problem in the qualitative theory of differential systems in the plane  is the problem of distinguish between a focus and a center. The following result was initiated by Poincar\'e and later rigorously established for analytic systems by Lyapunov.

\begin{theorem}[{Poincar\'e \cite{Poi81,Poi51}, Lyapunov \cite{Lya07}}]\label{Thp2}
Let  \eqref{S1} be a real planar analytic system with a singular point $(x_0,y_0)$ at which  the  linear part of the system  has purely imaginary eigenvalues. Then $(x_0,y_0)$ is a center if, and only if, there exists a non-constant local analytic first integral in a neighborhood of $(x_0,y_0)$. 
\end{theorem}

In view of Theorem~\ref{Thp2}, the existence of center-type equilibria for system
\eqref{Ga1} can only occur at isolated equilibria for which  linearization has
purely imaginary eigenvalues, that is, when $T=0$ and $D>0$.
The explicit expressions for the trace and determinant provided in
Theorem~\ref{AI} allow us to clarify when this situation may arise.

\begin{remark}
From Tables~\ref{betan0}--\ref{bet0} it follows that, for the equilibria
$\mathfrak p_2,\mathfrak p_3,\mathfrak p_{4,j},\mathfrak q_1,\mathfrak q_2$,
the condition $T=0$ implies $D<0$, and hence these equilibria are necessarily
saddles and cannot be centers in the real plane.

For the equilibria $\mathfrak p_{1,j}$, the condition $T=0$ does not determine
the sign of $D$, and purely imaginary eigenvalues may occur. However, because
$\mathfrak p_{1,j}$ is defined via an $m$-th root, it represents a real equilibrium
of system~\eqref{Ga1} only for those branches $j$ for which
\[
x_*=\sqrt[m]{\frac{\beta\gamma}{\alpha c-\gamma}}\in\mathbb{R}.
\]
Consequently, the center--focus problem is meaningful only for such real branches.
The present results do not address this case.
\end{remark}

\begin{problem}
Theorem~\ref{AI} provides sufficient conditions ensuring the nonexistence of
analytic first integrals in a neighborhood of isolated equilibria whenever
the nonresonance condition holds.
However, these results do not cover the borderline cases in which this condition
fails, nor the situations for which the linearization has purely imaginary eigenvalues.
A natural open problem is therefore to characterize the integrability properties
of system \eqref{Ga1} under such resonance conditions. In particular:
\begin{itemize}
\item[(i)]
Determine whether analytic or meromorphic first integrals may exist in neighborhoods
of resonant equilibrium points.
\item[(ii)]
For those real equilibria $\mathfrak p_{1,j}$ satisfying $T=0$ and $D>0$,
determine whether the local dynamics exhibits center-type behavior in the sense
of real planar systems.
\end{itemize}
\end{problem}

 Under the parametric conditions $\gamma=0$ and $m=1$, system~\eqref{Ga1} can be
reduced to a Riccati differential equation for $x$ as a function of $y$.
This reduction is obtained by eliminating the independent variable and leads,
after a standard transformation, to a second--order linear differential
equation with fundamental solutions  given in terms of Bessel functions.
The existence of explicit first integrals for system~\eqref{Ga1} in this case is
therefore a direct consequence of this Riccati structure, as shown in the
following theorem.

\begin{theorem}\label{SFunct}
Assume that $\gamma=0$, $m=1$, $r,\alpha,c,k\in\mathbb{C}\setminus\{0\}$ and $\beta\in\mathbb{C}$.  
Let
\[
A(x,y)=\frac{1}{c}\sqrt{\frac{r}{\alpha k}}\,y^{\frac12}
J'_{\nu}\!\left(\frac{2}{c}\sqrt{\frac{r y}{\alpha k}}\right)
-\frac{r}{\alpha c k}\left(x-\frac{k-\beta}{2}\right)
J_{\nu}\!\left(\frac{2}{c}\sqrt{\frac{r y}{\alpha k}}\right),
\]
\[
B(x,y)=\frac{1}{c}\sqrt{\frac{r}{\alpha k}}\,y^{\frac12}
Y'_{\nu}\!\left(\frac{2}{c}\sqrt{\frac{r y}{\alpha k}}\right)
-\frac{r}{\alpha c k}\left(x-\frac{k-\beta}{2}\right)
Y_{\nu}\!\left(\frac{2}{c}\sqrt{\frac{r y}{\alpha k}}\right),
\]
where $J_\nu$ and $Y_\nu$ denote the Bessel functions of the first and second kind.
Let $U\subset\mathbb{C}^2$ be an open set on which $A$ and $B$ are analytic and do not vanish simultaneously.
Then the function
\[
H(x,y)=\frac{A(x,y)}{B(x,y)}
\]
is a first integral of system~\eqref{Ga1} in $U$.
\end{theorem}

\begin{remark}[Global analytic extension on the Riemann surface]
Under the hypotheses of Theorem~\ref{SFunct}, the functions $A$ and $B$
involve the factors $y^{1/2}$, $J_{\nu}(z)$ and $Y_{\nu}(z)$, where
\[
z=\frac{2}{c}\sqrt{\frac{r y}{\alpha k}}.
\]
As $J_{\nu}$ and $J'_{\nu}$ are entire, while $y^{1/2}$ has a branch point at
$y=0$ and $Y_{\nu}$, $Y'_{\nu}$ possess a logarithmic singularity at $z=0$ together
with a branch cut (typically along the negative real axis in the $z$--plane),
the functions $A$ and $B$ are analytic precisely on
\[
\mathbb{C}^{2}\setminus S,
\]
where $S$ denotes the singularity set consisting of the hypersurface $\{y=0\}$
together with the preimage—under the map
$y\mapsto z=\frac{2}{c}\sqrt{\frac{r y}{\alpha k}}$ and the chosen branch of the square
root—of the branch cut of $Y_{\nu}$.

On $\mathbb{C}^{2}\setminus S$, the scalar first integral
\[
H(x,y)=\frac{A(x,y)}{B(x,y)}
\]
is in general multivalued due to the presence of branch points.
Choose a simply connected open set
\[
U_{0}\subset \mathbb{C}^{2}\setminus S
\]
on which a single branch of $H$ is fixed.
Analytic continuation of this branch along all paths in
$\mathbb{C}^{2}\setminus S$, glued together in the standard way,
produces a connected Riemann surface
\[
\pi:\mathcal{R}\longrightarrow \mathbb{C}^{2}\setminus S,
\]
where $\pi$ is the canonical projection.

The lifted function
\[
\widetilde H := H\circ \pi:\mathcal{R}\longrightarrow \mathbb{C}
\]
is then single--valued and holomorphic on $\mathcal{R}$, and it defines a global
analytic first integral of the lifted vector field on $\mathcal{R}$.
Each level set of $\widetilde H$ projects under $\pi$ onto a trajectory of the
original dynamical system on $\mathbb{C}^{2}\setminus S$.

Thus, although $H=\frac{A}{B}$ is multivalued on $\mathbb{C}^{2}$, it becomes a global
holomorphic first integral on its natural analytic--continuation Riemann surface
$\mathcal{R}$. This interpretation is consistent with the classical theory of
Liouvillian and Darboux-type first integrals, where multivalued invariants are
understood as global analytic objects defined on appropriate Riemann surfaces.
\end{remark}

%

\begin{remark}
The first integral $\dsty H=\frac{A}{B}$ is defined only up to Möbius transformations arising
from a change of the fundamental pair of solutions $(A,B)$.
Equivalently, one may regard the invariant quantity as the projective class
\[
H(x,y)=[\,A(x,y):B(x,y)]\in\mathbb{C}\mathbb{P}^1,
\]
which is well defined whenever $(A,B)\neq(0,0)$.
This projective formulation provides a coordinate-free description of the
invariant foliation associated with the Riccati reduction.

Moreover, although $H$ may be multivalued on $\mathbb{C}^{2}$ due to the presence
of branch points of the functions involved in $A$ and $B$, it becomes a
single-valued holomorphic map after lifting to the Riemann surface obtained by
analytic continuation along paths in $\mathbb{C}^{2}\setminus S$, where $S$ is
the singularity set of $A$ and $B$.
On this Riemann surface, the induced map into $\mathbb{C}\mathbb{P}^{1}$ defines
a global analytic first integral of the lifted vector field.
\end{remark}

%
%
%

Theorem~\ref{SFunct} provides, in particular, an explicit expression for an
analytic first integral in an open set $U\subset\mathbb{C}^{2}$ where the
functions $A$ and $B$ are analytic and do not vanish simultaneously.
This includes neighborhoods of nonisolated equilibrium points of the form
$(0,y)$ with $y\neq 0$, provided such points belong to $U$.
At these equilibria the Jacobian matrix of the linear part of the vector field
$L_{v_1}$ is singular, and therefore the existence of an analytic first integral
cannot be inferred from Poincar\'e’s theorem.
Theorem~\ref{SFunct} thus supplies analytic first integrals in a regime not
covered by the classical nonresonant linearization theory, complementing the
local analysis of analytic solutions near degenerate equilibria.

Besides analytic continuation, there is a recent alternative global approach due
to Gin\'e and Sinelshchikov (see, e.g., \cite{GinSi25}), where one constructs
global transcendental first integrals from a finite  family of   invariant (possibly
transcendental) curves whose cofactors satisfy a nontrivial linear relation.  If
such invariant curves were identified for the present system~\eqref{Ga1}, this
method  produces an expression for  non Liouvillian  first integrals—potentially written in terms of different special or transcendental functions for the special case studied in Theorem \ref{SFunct}—thus raising
the interesting problem of determining whether the system admits  such  invariant transcendental  curves
suitable for this construction.

\section{Liouvillian integrability} \label{S3}

Before entering the main analysis, we state an auxiliary algebraic lemma, which is needed in the proof of Theorem \ref{Liov} and plays a key role in excluding certain polynomial solutions.

\begin{lemma}\label{PolS}
Suppose  $(\alpha,r,c,k,\gamma,\beta)\in A$, $\dsty 
X(x)=\prod_{x_i\in I}(x-x_i)^{n_i}$, where $n_i\in\mathbb{Z}_{\ge 0}$ denotes the corresponding  multiplicity of the root $x_i$ of $X$, and assume $X(x)\neq x^n$ for any $n\in\mathbb{Z}_{\ge 0}$.  
Let $R_m$ be any polynomial of degree $m$.  
Then the differential equation
\begin{eqnarray}\label{i0+1}
-\frac{r}{k}\Pi_{m+2}(x)\, w^{\prime}(x)
+\left(
\frac{r}{k}\Pi_{m+2}(x)\frac{X^{\prime}(x)}{X(x)}
+R_m(x)
\right) w(x)
= -n\alpha x^{m-1}X(x) + \alpha x^m X^{\prime}(x),
\end{eqnarray}
has no polynomial solutions $w$.
\end{lemma}

 \proof

Write   equation  \eqref{i0+1} as 
\begin{eqnarray*}
 T[w]=-n\alpha x^{m-1}X(x)+\alpha x^mX^{\prime}(x),
\end{eqnarray*}
 where $T$ is the operator  acting on the vector  space   of  polynomials with complex coefficients  defined by the left-hand side of \eqref{i0+1}. A straightforward calculation shows that the main coefficient of $T[x^j], j\in  \ZZ_{j\geq 0}$ (that is, the coefficient of its highest–degree term) is $-\frac{r}{k}(-j+n_1+\ldots+n_I)$. Hence,  if $j\neq n_1+\ldots+n_I$,
\begin{eqnarray}\label{K}
\deg (T[x^j])=  j+m+1.
\end{eqnarray}

Let $n_1$ be the multiplicity of the root $x=0$ of $X$. Because $(\alpha,r,c,k,\gamma,\beta)\in A$, at least one $n_i(i>1)$ is nonzero. Therefore, 
$$\dsty X(x)= x^{n_1}\prod_{i=2}^{m+2}(x-x_i)^{n_i}.$$ 

If  $w$  is a polynomial solution, factor out all forced zeros imposed by the equation at each root of  $X$. One obtains the representation
\begin{eqnarray}\label{p}
w(x)=  x^{m-1}\prod_{x_i\in I}(x-x_i)^{n_i-1}u(x)
\end{eqnarray}
where $u$ is a polynomial.

Computing the degree of the right-hand side, one gets
\begin{eqnarray}\label{Altern}
\deg [ -n\alpha x^{m-1}X(x)+\alpha x^mX^{\prime}(x)]=
\begin{cases}
m+n_1+\ldots+n_I -1,  & \mbox{when} \quad n\neq n_1+\ldots+n_{I},\\
m+n_1+\ldots+n_I -2,  &  \mbox{when} \quad  n= n_1+\ldots+n_{I}.
\end{cases} 
\end{eqnarray}

Substitution of  \eqref{p} into \eqref{i0+1}, together with the use of  \eqref{K} and \eqref{Altern}, yields
\begin{itemize}
\item If $n\neq n_1+\ldots +n_I$, then  $|I|\in \{m+1,m+2\}$  and  $\deg [u]\leq 1$.
\item If $n= n_1+\ldots +n_I$, then  $|I|=m+2$  and  $\deg [u]=0$.
\end{itemize}
 
If  $\deg [u]=0$, the representation of $w$  above shows that $w$ cannot be a polynomial because it forces an inconsistent multiplicity at 
$x=0$.

Therefore we are left with the only remaining case:
$$|I|=m+2, n\neq n_1+\ldots +n_I, u(x)=ax+b. $$

  To prove that  $u(x)=ax+b$ does not define  a polynomial solution, assume that $|I|=m+2$. Then  $n_1\geq 1$, and  can write the right-hand side of \eqref{i0+1} as 
\begin{eqnarray}\label{indepT}
-n\alpha x^{m-1}X(x)+\alpha x^mX^{\prime}(x)=x^{m+n_1-1}\prod_{i=2}^{m+2}(x-x_i)^{n_i-1} \widetilde{\Pi}_3(x),
\end{eqnarray}
where $\widetilde{\Pi}_3$ is a polynomial of degree $m+1$ such that $\widetilde{\Pi}_3(0)\neq 0$.

In view of  \eqref{indepT},  the substitution 
\begin{eqnarray}\label{subst}
\dsty w(x)= x^{m-1}\prod_{i=1}^{m+2}(x-x_i)^{n_i-1}u(x)
\end{eqnarray}
 transforms  the differential equation \eqref{i0+1} into 
\begin{eqnarray}\label{i0+1,0}
x\widetilde{\Pi}_{1}(x) u^{\prime}(x)+( \widetilde{\Pi}_2(x)+ \widetilde{\Pi}(x))u(x)=
x \prod_{i=2}^{m+2}(x-x_i)\left(\alpha(n_1-n)+\alpha x \frac{\left( \prod_{i=2}^{m+2}(x-x_i)^{n_i}\right)^{\prime}}{\prod_{i=2}^{m+2}(x-x_i)^{n_i}}\right),
\end{eqnarray}
 where 
 \begin{eqnarray*}
\widetilde{\Pi}_1(x)&=& - \frac{r}{k}\prod_{i=2}^{m+2}(x-x_i),\\
 \widetilde{\Pi}_2(x)&=& \frac{r}{k}\Pi_{m+2}(x) \frac{X^{\prime}(x)}{X(x)}+R_m(x), \\
 \widetilde{\Pi}(x)&=& -\frac{r}{k} \left((n_1+m-2)\prod_{i=2}^{m+2}(x-x_i)+  x  \prod_{i=2}^{m+2}(x-x_i)  \frac{\left( \prod_{i=2}^{m+2}(x-x_i)^{n_i-1}\right)^{\prime}}{\prod_{i=2}^{m+2}(x-x_i)^{n_i-1}   }   \right).
 \end{eqnarray*} 
 
The substitution of $u(x)=ax+b$ into \eqref{i0+1,0} and comparison of the coefficients of $x^{m+2}$, together with  $r\neq 0$, yield
$$a=\frac{k\alpha(n_1+\ldots +n_{m+2}-n)}{2r }.$$
On the other hand, by setting $x=0$, one gets 
 \begin{eqnarray*} 
b=
\begin{cases}
0, & \text{if } R_m(0)+r\beta(m-2)\neq 0,\\
\text{const.}, & \text{if } R_m(0)+r\beta(m-2)=0.
\end{cases} 
\end{eqnarray*}
  By assuming  that  $R_m(0)+r\beta (m-2)\neq 0$,  one obtains from \eqref{subst} that the polynomial solution to \eqref{i0+1} is 
$$\dsty w(x)= \frac{k\alpha(n_1+\ldots +n_{m+2}-n)}{2r }x^m\prod_{i=1}^{m+2}(x-x_i)^{n_i-1},$$ 
 which is not possible, because the multiplicity of  $x=0$ of the polynomial solution of  \eqref{i0+1}  is $m-1$. 
 
 If $R_m(0)+r\beta (m-2)= 0$, then the same substitution yields that any polynomial solution would have the form
 $$ \dsty w(x)=-\frac{k\alpha(n_1+\ldots +n_{m+2}-n)}{2r }x^m\prod_{i=1}^{m+2}(x-x_i)^{n_i-1}+\mbox{const.}x^{m-1}\prod_{i=1}^{m+2}(x-x_i)^{n_i-1}.  $$
This  implies that $\dsty x^{m-1}\prod_{i=1}^{m+2}(x-x_i)^{n_i-1}$ is in the kernel of the operator $T$, which is not possible by relation  \eqref{K}. Therefore no polynomial solution exists in this subcase either. This completes the proof. \lqqd

  \proof \emph{(Of Theorem \ref{Liov})}
 
Let   $f_l(x,y)=0, l=1,2$ be an  algebraic invariant curve of  the vector field $L_{v_l}$. Write  
\begin{eqnarray}\label{f}
  \dsty f_l(x,y)=\sum_{j=0}^{M}X_j(x)y^j,
\end{eqnarray}  
where $X_j\in\CC[x]$ and $X_M\not\equiv 0$. For brevity we omit the index $l$ when no confusion arises.

Because $L_{v_l}[x^i y^j]$ increases the power of $y$ by at most one, any cofactor $\mathfrak c_l(x,y)$ of an invariant polynomial must satisfy $\deg_y\mathfrak c_l\le 1$. Hence we may write
\[
\mathfrak c_l(x,y)=P(x)+Q(x)\,y.
\]

Substitution of \eqref{f} into the invariance relation $L_{v_l}[f_l]=\mathfrak c_l f_l$ and the comparison of the coefficients of powers of $y$ yield the following identities.

For $L_{v_1}$ (and $M>0$) we obtain
\begin{multline}\label{HS}
-\frac{r}{k}\Pi_{m+2}(x)  X_{0}^{\prime}(x) +\sum_{j=1}^{M}\left(-\frac{r}{k}\Pi_{m+2}(x)X_{j}^{\prime}(x)-\alpha x^{m}X^{\prime}_{j-1}(x)-(j\gamma\beta-j(\alpha c-\gamma)x^m)X_{j}(x)\right)y^{j} - \alpha x^mX^{\prime}_{M}(x)y^{M+1}=\\
P(x)X_0(x)+\sum_{j=1}^{M}(P(x)X_j(x)+Q(x)X_{j-1}(x))y^j+Q(x)X_{M}(x)y^{M+1}, \quad M>0.
\end{multline} 
For $L_{v_2}$ (and $M>0$) we obtain
\begin{multline}\label{HS2}
-\frac{r}{k}x(x-k)  X_{0}^{\prime}(x) +\sum_{j=1}^{M}\left(-\frac{r}{k}x(x-k)  X_{j}^{\prime}(x)-\alpha X^{\prime}_{j-1}(x)+j(\alpha c-\gamma)X_{j}(x)\right)y^{j} - \alpha  X^{\prime}_{M}(x)y^{M+1}=\\
P(x)X_0(x)+\sum_{j=1}^{M}(P(x)X_j(x)+Q(x)X_{j-1}(x))y^j+Q(x)X_{M}(x)y^{M+1}, \quad M>0.
\end{multline} 
If $M=0$, the summation symbol  is omitted.

\vspace{6pt}\noindent\textbf{Case $M=0$.}  Then $f=X_0(x)$, and comparison of terms of equal degree in $y$ in \eqref{HS} yields for $L_{v_1}$
\begin{eqnarray}
\begin{aligned}\label{M=0,1}
X_0(x)&=\mbox{const.} x^n,\\
P(x)&=-\frac{nr}{k}(x-k)(x^m+\beta), Q(x)=-n\alpha x^{m-1},
\end{aligned}
\end{eqnarray}
with $n\in\ZZ_{>0}$. For $L_{v_2}$ one obtains only the trivial invariant (constant), 
\begin{eqnarray}
\begin{aligned}\label{M=0,2}
X_0(x)&=\mbox{const.},\\
P(x)&=0, Q(x)=0.
\end{aligned}
\end{eqnarray} 
Hence no nontrivial algebraic curve with $M=0$ appears for $L_{v_2}$.

\vspace{6pt}\noindent\textbf{Case $M>0$.} Let $j_0\ge0$ be the minimal index with $X_{j_0}\not\equiv 0$. 
 
\smallskip\noindent If $j_0=M$,  comparison of  the coefficients of $y^M$ and $y^{M+1}$ in \eqref{HS} yields the system
\begin{eqnarray*} 
 \frac{r}{k}\Pi_{m+2}(x) X_{j_0}^{\prime}(x)+(P(x)+(j_0\gamma\beta-j_0(\alpha c-\gamma)x^m))X_{j_0}(x)&=&0, \\
  \alpha x^mX^{\prime}_{j_0}(x) +Q(x)X_{j_0}(x)&=&0
\end{eqnarray*}
which implies 
\begin{eqnarray}
\begin{aligned}\label{Q0}
X_{j_0}(x)&= \mbox{const.}x^n,\\
  P(x)&=-\frac{nr}{k}(x-k)(x^m+\beta)-(j_0\gamma\beta-j_0(\alpha c-\gamma)x^m), Q(x)=-n\alpha x^{m-1}
\end{aligned}
\end{eqnarray}
 where   $n\in \ZZ_{\geq 0}$.
 
For $L_{v_2}$ an analogous comparison yields
\begin{eqnarray*} 
 -\frac{r}{k} x(x-k) X_{j_0}^{\prime}(x)+(P(x)-j_0(\alpha c-\gamma))X_{j_0}(x)&=&0, \\
  \alpha X^{\prime}_{j_0}(x) +Q(x)X_{j_0}(x)&=&0
\end{eqnarray*}
which implies
\begin{eqnarray}
\begin{aligned}\label{Q0,2}
X_{j_0}(x)&= \mbox{const.},\\
  P(x)&=j_0(\alpha c-\gamma); Q(x)=0.
\end{aligned}
\end{eqnarray}

\smallskip\noindent If $j_0<M$, comparison of  the coefficient of $y^{j_0}$ in \eqref{HS} yields
\begin{eqnarray}\label{Eq1}
-\frac{r}{k}\Pi_{m+2}(x) \frac{X_{j_0}^{\prime}(x)}{X_{j_0}(x)}=P(x)-(j_0\gamma\beta-j_0(\alpha c-\gamma)x^m).
\end{eqnarray}
Thus,
\begin{eqnarray}
\begin{aligned}\label{XP}
X_{j_0}(x)&=\mbox{const.} \prod_{x_i\in I}(x-x_i)^{n_i}, \quad n_i\in \ZZ_{\geq0},\\
P(x)&=-\frac{r}{k}\Pi_{m+2}(x) \frac{X_{j_0}^{\prime}(x)}{X_{j_0}(x)}-j_0\gamma\beta+j_0(\alpha c-\gamma)x^m.
\end{aligned}
\end{eqnarray}

A comparison of  the coefficients of $y^{M+1}$ yields 
$$-\alpha x^m\frac{X^{\prime}_{M}(x)}{X_{M}(x)}=Q(x),$$
so, 
\begin{eqnarray}
\begin{aligned}\label{Q}
X_{M}(x)&= \mbox{const.} x^n, n\in \ZZ_{\geq0}, \\
Q(x)&=-n\alpha x^{m-1}.
\end{aligned}
\end{eqnarray}

Comparison of  the coefficient of $y^{j_0+1}$ in \eqref{HS},  after substitution of \eqref{XP} and \eqref{Q},  yields the linear inhomogeneous ODE,
\begin{eqnarray}\label{Lem1}
-\frac{r}{k}\Pi_{m+2}(x) X_{j_0+1}^{\prime}(x)+\left( \frac{r}{k}\Pi_{m+2}(x) \frac{X^{\prime}_{j_0}(x)}{X_{j_0}(x)}-\gamma\beta+(\alpha c-\gamma)x^m\right)X_{j_0+1}(x)=-n\alpha x^{m-1}X_{j_0}(x)+\alpha x^mX^{\prime}_{j_0}(x). 
\end{eqnarray}

If $j_0=M-1$ and the right-hand side of \eqref{Lem1} vanishes, then $X_{j_0}=x^n$. In this exceptional situation, substitution into \eqref{Lem1} leads to an identity that cannot hold unless the parameter relation in $B$ occurs; therefore it is impossible when $(\alpha,r,c,k,\gamma,\beta)\notin B$.

If $j_0=M-1$ and the right-hand side of \eqref{Lem1} does not vanish, Lemma \ref{PolS} (applied with the appropriate choice of polynomials) implies that no polynomial solution $X_{j_0+1}$ exists when $(\alpha,r,c,k,\gamma,\beta)\in A$.

If $j_0<M-1$ and $X_{j_0}$ satisfies the exceptional relation $-n\alpha x^{m-1}X_{j_0}+\alpha x^mX_{j_0}'=0$, replace $j_0$ by $j_0+1$ and repeat the above argument. The iteration either reaches an exceptional case (which forces $(\alpha,r,c,k,\gamma,\beta)\in B$) or, at some step, falls into the nonexceptional case handled by Lemma \ref{PolS}. In the latter situation Lemma \ref{PolS} again yields nonexistence of polynomial solutions for the corresponding $X_{j_0+1}$ when $(\alpha,r,c,k,\gamma,\beta)\in A$  and $\beta\neq0$. Thus no polynomial invariant arises.

An analogous analysis holds for $L_{v_2}$.   For  $j_0<M$,  comparison of  the coefficients  of $y^{j_0}$ in \eqref{HS2} yields  that $P$ and $X_{j_0}$ satisfy
\begin{eqnarray}\label{Eq1,2}
-\frac{r}{k}x(x-k) \frac{X_{j_0}^{\prime}(x)}{X_{j_0}(x)}=P(x)-j_0(\alpha c-\gamma).
\end{eqnarray}
Hence,
\begin{eqnarray}
\begin{aligned}\label{XP,2}
X_{j_0}(x)&=\mbox{const.}x^{n_1}(x-k)^{n_2},\\
P(x)&=-\frac{r}{k}x(x-k)\frac{X_{j_0}^{\prime}(x)}{X_{j_0}(x)}+j_0(\alpha c-\gamma).
\end{aligned}
\end{eqnarray}

Comparison of    the coefficients  associated to $y^{M+1}$ implies
$$\alpha X^{\prime}_{M}(x)+ Q(x) X_{M}(x)=0.$$
Hence,
\begin{eqnarray}
\begin{aligned}\label{Q,2}
X_{M}(x)&=\mbox{const.}\neq 0, \\
Q(x)&=0,P(x)=M(\alpha c-\gamma).
\end{aligned}
\end{eqnarray}

By virtue of  \eqref{XP,2} and \eqref{Q,2}, we obtain from \eqref{HS2} the coefficient associated with  $y^{j_0+1}$
\begin{eqnarray}\label{Lem2}
-\frac{r}{k}x(x-k)  X_{j_0+1}^{\prime}(x)+\left(\frac{r}{k}x(x-k)\frac{X_{j_0}^{\prime}(x)}{X_{j_0}(x)}+\alpha c-\gamma \right)X_{j_0+1}(x)= \alpha X^{\prime}_{j_0}(x).
\end{eqnarray}

 For  $j_0=M-1$, if  $X_{j_0}=\mbox{const.}$   we have that  equation \eqref{Lem2} implies
$$    \alpha c-\gamma=0,$$  
 which is not  possible by hypothesis because $(\alpha,r,c,k,\gamma,\beta)\notin B$.

If  $j_0=M-1$ and $X_{j_0}\neq \mbox{const.}$,  take $n=0, R_m(x)= (\alpha c-\gamma)x^m$, and divide by $x^m$ in Lemma \ref{PolS}. It follows  that  there are no polynomial solutions to \eqref{Lem2}.

If   $j_0<M-1$, we may assume  $X_{j_0}\neq \mbox{const.}$ as well, otherwise, equation \eqref{Lem2} for $X_{j_0+1}$ reduces to 
$$-\frac{r}{k}x(x-k) w^{\prime}(x)+\left(\alpha c-\gamma \right)w(x)=0,$$
whose general solution is $\dsty w(x)= \mbox{const.}\left( \frac{x-k}{x } \right)^{ \frac{ \alpha c - \gamma }{r} }$  which  is not a polynomial because $(\alpha,r,c,k,\gamma,\beta)\in A\setminus B$. Hence, by taking $n=0, R_m(x)= (\alpha c-\gamma)x^m$, and dividing by $x^m$ in Lemma \ref{PolS} we deduce that  there is no polynomial solution to \eqref{Lem2}.

Collecting the above results we deduce that the only possible irreducible Darboux polynomials are those obtained in the $M=0$ and $j_0=M$ cases. This proves the theorem. \lqqd

\begin{lemma}\label{FPS,3}
Assume  $(\alpha,r,c,k,\gamma,\beta)\in A\setminus B$, and $m\geq 1$.  Let    $L_{v_1}$ and $  L_{v_2}$ be as in \eqref{vf} and $L_1:=L_{v_1}, L_{2}:=x^m L_{v_2}$.  Suppose $L_{v_1}$ and $  L_{v_2}$ admit exponential factors
\[
F_i=e^{f_i/g_i},\qquad f_i,g_i\in\CC[x,y],\quad i=1,2,
\]
with $f_i,g_i$ coprime. Then    $g_i$ is the Darboux polynomial  given  by \eqref{f(x,y)} and 
\begin{align}\label{A1}
\deg[f_i]\leq n_1+n_2+1,
\end{align}
where  $n_1$ and  $n_2$ are   the exponents of the variables $x$ and   $y$ respectively of  $g_i$.
\end{lemma} 

 \proof
By Proposition \ref{Link}, each denominator $g_i$ is a Darboux polynomial of the corresponding polynomial vector field, so  by Theorem \ref{Liov},  $g_1$ and $g_2$ have  the expression given by \eqref{f(x,y)} with   associated cofactors $K_{g_1}=\mathcal{K}_{g_1}$ and $K_{g_2}= x^m \mathcal{K}_{g_2}$ respectively. By relation  \eqref{def[f]}, 
 \begin{eqnarray*} 
 L_i[f_i]-f_iK_{g_i}=g_iK_{F_i},
 \end{eqnarray*}
where   $K_{F_i}$  denotes  the cofactor of  $F_i$.
 
 We have that 
\begin{eqnarray*}
\deg[L_{i}[f_i]]&= &\deg[f_i]+ m+\kappa_1,
\end{eqnarray*} 
where $\kappa_1 \in \{0,1\}$.

 From the expressions for   \eqref{f(x,y)}   we deduce that, 
\begin{eqnarray*}
\deg[f_iK_{g_i}]=\deg [f_i]+m+\kappa_2,
\end{eqnarray*} 
where $\kappa_2 \in \{0,1\}$.

Because $\deg[K_{F_i}]\leq m+1$, one also has  
\begin{eqnarray*}
\deg[g_iK_{F_i}]&\leq &n_1+n_2+m+1.
\end{eqnarray*} 
  Hence, from the preceding relations we obtain the desired inequality.
 \lqqd

  \proof \emph{(Of Theorem \ref{Liov2})}
  
We prove that there are no Darbouxian integrating factors for \(L_{v_i}\); it then follows immediately that \(L_{v_i}\) is not Liouvillian integrable, by \cite[Th.~8.17, p.~228]{DFL} and \cite[Th.~3.11, p.~134]{Zh17}. 
To give a unified treatment of the proof, we replace \(L_{v_2}\) by \(x^m L_{v_2}\) and denote by \(L_w\) the unique vector field that coincides with \(L_{v_1}\) when \(\beta\neq 0\) and with \(L_{v_2}\) when \(\beta=0\).
Assume that
\begin{equation}\label{I1}
I(x,y)=\prod_{j=1}^{J_1}\pi_j^{\mu_j}(x,y) e^{R(x,y)}, 
\qquad 
R(x,y)=\frac{g(x,y)}{h(x,y)}, 
\quad g,h\in \mathbb{C}[x,y],\ \mu_j\in \mathbb{C},
\end{equation}
is a Darbouxian integrating factor.
By \cite[Prop.~6]{LlVZ16}, each \(\pi_j\) is an irreducible Darboux polynomial.
Hence, by Theorem \ref{Liov}, we can write \eqref{I1} in the form
\[
I(x,y)=x^{\lambda_1}y^{\lambda_2}e^{R(x,y)}.
\]
Moreover, by the same proposition, \(h\) is also a Darboux polynomial.
By Theorem \ref{Liov}, we can express \(h(x,y)=x^{n_1}y^{n_2}\), with
\(n_1,n_2\in \mathbb{Z}_{\geq 0}\)  when \(\beta\neq 0\), and
\(h(x,y)=x^{m}y^{n_2}\), with \(n_2\in \mathbb{Z}_{\geq 0}\), when \(\beta=0\).
Finally, if \(I=\prod_{j=1}^{J_1}\pi_j^{\mu_j} e^{R}\) is an integrating factor of a polynomial
vector field \(L=p\partial_x+q\partial_y\) and the \(\pi_j\) are Darboux polynomials,
then \(e^{R}\) is necessarily an exponential factor. Indeed, 
 the representation of $I$ in the form  \(I=Pe^{R}\), where \(P=\prod_{j=1}^{J_1}\pi_j^{\mu_j}\), together with the identity 
\(L(I)=-I\,\operatorname{div}(p,q)\), yields
\(L(e^{R})=K e^{R}\) with \(K\in\mathbb{C}[x,y]\).

Put $\dsty g(x,y)=\sum_{n=0}^{N}a_n(x)x^n$.  Then 
\begin{eqnarray}\label{Rdev}
R(x,y)=\sum_{n=-n_2}^{N-n_2}\frac{a_{n+n_2}(x)}{x^{n_1}}y^n, \end{eqnarray}  
where $n_1=m$ when $\beta=0$.

Equation \eqref{LH} for  the  integrating factor associated to the vector field $L_{w}$, after canceling the term $x^{\lambda_1-1}y^{\lambda_2-1}e^{R(x,y)}$, can be expressed as 
\begin{eqnarray}\label{IntFE}
\sum_{n=0}^{N+1}D_{n}(x)y^{n-n_2+1}=  \left(\frac{r}{k}\Pi_{m+2}'(x) + \gamma\beta - (\alpha c - \gamma)x^m\right)xy + \alpha m x^{m} y^2.
\end{eqnarray}
Notice  that
\begin{eqnarray}\label{F1}
N-n_2\geq 0, 
\end{eqnarray}
because otherwise,  for $N-n_2< 0$, from equation \eqref{IntFE} we deduce that the coefficient associated to $y^2$ in the right-hand side must be null, which is not possible because $(\alpha,r,c,k,\gamma,\beta)\in A$, that is $\alpha\neq 0$.

After some   calculations we have
\begin{eqnarray}\label{Dn}
\begin{aligned}
 D_{0}(x)&=  \begin{cases}
 \dsty P_{m+2}(x)-xL\left[\frac{a_{0}(x)}{x^{n_1}}\right],  &  n_2=0,\\
  \dsty -xL\left[\frac{a_{0}(x)}{x^{n_1}}\right],  &  n_2>0,
 \end{cases} \\
 D_{n}(x)&=  \begin{cases}
\dsty  -\alpha x^{m+1} \left(\frac{a_{n-1}(x)}{x^{n_1}}\right)^{\prime}-xL\left[\frac{a_{n}(x)}{x^{n_1}}\right], & 0<n\leq N, \; n\neq n_2, \; n\neq  n_2+1,\\
\dsty  -\alpha x^{m+1}\left(\frac{a_N(x)}{x^{n_1}}\right)^{\prime},&  n=N+1, \, n_2\neq N,\\
\dsty  -\alpha\lambda_1 x^m-\alpha x^{m+1}\left(\frac{a_N(x)}{x^{n_1}}\right)^{\prime},&  n=N+1, \, n_2= N,
 \end{cases} \\
  D_{n}(x)&=  \begin{cases}
  \dsty P_{m+2}(x) -\alpha x^{m+1} \left(\frac{a_{n-1}(x)}{x^{n_1}}\right)^{\prime}-xL\left[\frac{a_{n}(x)}{x^{n_1}}\right], &  n= n_2,\, n_2\leq N,\\
\dsty -\alpha \lambda_1x^m-xL\left[\frac{a_{n}(x)}{x^{n_1}}\right] -\alpha x^{m+1}\left(\frac{a_{n-1}(x)}{x^{n_1}}\right)^{\prime},&  n=n_2+1,\,  n_2+1\leq N,
 \end{cases} \\
L\left[\frac{a_{n}(x)}{x^{n_1}}\right]&= \frac{r}{k}\Pi_{m+2}(x)\left(\frac{a_{n}(x)}{x^{n_1}}\right)^{\prime}+(n-n_2)( \gamma\beta- (\alpha c-\gamma)x^m )\frac{a_{n}(x)}{x^{n_1}},\\
P_{m+2}(x)&=- \frac{\lambda_1r}{k}\Pi_{m+2}(x)- \lambda_2(\gamma\beta x-(\alpha c-\gamma)x^{m+1}).
\end{aligned}
 \end{eqnarray}

In view of  \eqref{Dn},  under assumption \eqref{F1}, we must solve in polynomials  the set of differential equations
\begin{eqnarray} 
\left. \begin{array}{r}\label{Set1}
 \dsty P_{m+2}(x)-xL\left[\frac{a_{0}(x)}{x^{n_1}}\right]= \left(\frac{r}{k}\Pi_{m+2}'(x) + \gamma\beta - (\alpha c - \gamma)x^m\right)x, \quad n=0,\\
\dsty  -\alpha \lambda_1x^m -\alpha x^{m+1}\left(\frac{a_{0}(x)}{x^{n_1}}\right)^{\prime}=\alpha m x^{m}, \quad n=1
\end{array} \right\} N=0   \\
\left. \begin{array}{r} \label{Set2}
 \dsty P_{m+2}(x)-xL\left[\frac{a_{0}(x)}{x^{n_1}}\right]= \left(\frac{r}{k}\Pi_{m+2}'(x) + \gamma\beta - (\alpha c - \gamma)x^m\right)x, \quad n=0,\\
\dsty -\alpha \lambda_1x^m  -\alpha x^{m+1} \left(\frac{a_{n-1}(x)}{x^{n_1}}\right)^{\prime}-xL\left[\frac{a_{n}(x)}{x^{n_1}}\right]=\alpha m x^{m},  \quad n=1,\\
\dsty  -\alpha x^{m+1} \left(\frac{a_{n-1}(x)}{x^{n_1}}\right)^{\prime}-xL\left[\frac{a_{n}(x)}{x^{n_1}}\right]=0, \quad 1<n\leq N,N\neq 1,\\
\dsty  -\alpha x^{m+1}\left(\frac{a_N(x)}{x^{n_1}}\right)^{\prime}=0, \quad n=N+1
\end{array} \right\} N>0
 \end{eqnarray}
for $n_2=0$. For $n_2>0$ we obtain 
\begin{eqnarray}
 \left. \begin{array}{r}\label{Set3}
\dsty L\left[\frac{a_{0}(x)}{x^{n_1}}\right]=0, \quad n=0,\\
 \dsty  -\alpha x^{m+1} \left(\frac{a_{n-1}(x)}{x^{n_1}}\right)^{\prime}-xL\left[\frac{a_{n}(x)}{x^{n_1}}\right]=0,  \quad 0<n\leq N, n\neq n_2,n_2+1,\\
\dsty P_{m+2}(x)-\alpha x^{m+1} \left(\frac{a_{n-1}(x)}{x^{n_1}}\right)^{\prime}-xL\left[\frac{a_{n}(x)}{x^{n_1}}\right]= \left(\frac{r}{k}\Pi_{m+2}'(x) + \gamma\beta - (\alpha c - \gamma)x^m\right)x, \quad n=n_2,\\
\dsty  -\alpha \lambda_1x^m-\alpha x^{m+1} \left(\frac{a_{n-1}(x)}{x^{n_1}}\right)^{\prime}-xL\left[\frac{a_{n}(x)}{x^{n_1}}\right]=\alpha mx^m, \quad n=n_2+1,\\
\dsty   -\alpha x^{m+1}\left(\frac{a_N(x)}{x^{n_1}}\right)^{\prime}=0, \quad n=N+1
 \end{array} \right\} n_2<N 
  \end{eqnarray}
  \begin{eqnarray}
 \left. \begin{array}{r}\label{Set4}
 \dsty L\left[\frac{a_{0}(x)}{x^{n_1}}\right]=0, \quad n=0,\\
  \dsty  -\alpha x^{m+1} \left(\frac{a_{n-1}(x)}{x^{n_1}}\right)^{\prime}-xL\left[\frac{a_{n}(x)}{x^{n_1}}\right]=0,  \quad 0<n < N,\\
 \dsty P_{m+2}(x)-\alpha x^{m+1} \left(\frac{a_{n-1}(x)}{x^{n_1}}\right)^{\prime}-xL\left[\frac{a_{n}(x)}{x^{n_1}}\right]= \left(\frac{r}{k}\Pi_{m+2}'(x) + \gamma\beta - (\alpha c - \gamma)x^m\right)x, \quad n=N,\\
\dsty -\alpha \lambda_1x^m  -\alpha x^{m+1}\left(\frac{a_N(x)}{x^{n_1}}\right)^{\prime}=\alpha mx^m, \quad n=N+1. 
   \end{array} \right\} n_2=N.
 \end{eqnarray}

For system  \eqref{Set1}, it follows immediately  from  the last equation that $\lambda_1=-m$ and $\dsty \frac{a_{0}(x)}{x^{n_1}}=\mbox{const.}$, which is incompatible with the first equation.

Next, we consider systems \eqref{Set2} and \eqref{Set3}.   Let $N-n_2=\delta\geq 1$. For $n=N+1>n_2$ we have 
$$-\alpha x^{m+1}\left(\frac{a_N(x)}{x^{n_1}}\right)^{\prime}=0,$$
which implies that $a_N(x)=\mbox{const.}x^{n_1} \neq 0$. By   Lemma \ref{FPS,3}, one has    $\deg[a_N(x)y^{N}]\leq n_1+n_2+1$, which   implies that   $\delta=1$. Let us  consider

\begin{itemize}
\item  The case $n_2>0$.   From \eqref{Set3}, for $n=n_2+1=N$,  one has that $a_{N-1}$ is the polynomial  solution of  
$$\dsty  -\alpha \lambda_1x^m-\alpha x^{m+1}\left(\frac{a_{N-1}(x)}{x^{n_1}}\right)^{\prime}-xL\left[\mbox{const.}\right]=\alpha mx^m ,$$  
because $\dsty \frac{a_{N-1}(x)}{x^{n_1}}=\mbox{const.}$ is not a solution, it follows that  $\lambda_1=-m$. The equation then simplifies to 
$$-\alpha  \left(\frac{a_{N-1}(x)}{x^{n_1}}\right)^{\prime}=(-\frac{\gamma\beta}{x^m}+\alpha  c-\gamma)\mbox{const.}, $$
which has polynomial solution if and only if $\gamma\beta=0$. In that case,  $\dsty a_{N-1}(x)=\mbox{const.}_1x^{n_1+1}+\mbox{const.}_2x^{n_1}$. For $n=n_2$, the equation that defines $a_{N-2} $ reduces  to 
$$\frac{mr}{k}(x-k)(x^m+\beta) +\lambda_2((\alpha c-\gamma)x^{m})-\alpha x^{m}\left(\frac{a_{N-2}(x)}{x^{n_1}}\right)^{\prime}-\frac{r}{k}\Pi_{m+2}(x)\frac{\alpha  c-\gamma}{\alpha}\mbox{const.}_1=  \frac{r}{k}\Pi_{m+2}'(x)  - (\alpha c - \gamma)x^m,$$
whose solution can be expressed as 
$$\frac{a_{N-2}(x)}{x^{n_1}}=-\frac{r(\alpha c-\gamma)}{3k\alpha^2}\mbox{const.}_1x^3+\mbox{lower order terms}.$$
By hypothesis, $\gamma\beta$, and $\alpha c-\gamma$ do not vanish simultaneously, and $r\neq 0$. This implies that $\deg[a_{N-2}]=n_1+3$,  which is not possible by  Lemma \ref{FPS,3}. Therefore,   \eqref{Set3} is not solvable by polynomials. 

\item The case $n_2=0$.  From \eqref{Set2}, for $n=N$,  one has that $a_{N-1}$ is the polynomial  solution of 
\begin{equation*} 
\begin{aligned} 
 \dsty  -\alpha x^{m} \left(\frac{a_{N-1}(x)}{x^{n_1}}\right)^{\prime}-L\left[\mbox{const.} \right]&=0, \quad  N>1,\\
 \dsty -\alpha \lambda_1x^m -\alpha x^{m+1} \left(\frac{a_{N-1}(x)}{x^{n_1}}\right)^{\prime}-xL\left[\mbox{const.} \right]&=\alpha mx^m, \quad N=1, 
\end{aligned} 
\end{equation*} 
  which  implies that   $\dsty a_{N-1}(x)=\frac{N(\alpha  c-\gamma)}{\alpha}\mbox{const.}x^{n_1+1}+\mbox{const.}_1x^{n_1}$ and  $\lambda_1=-m, \gamma\beta=0$ if $m>1$, and $\dsty\lambda_1=\frac{\gamma\beta-\alpha }{\alpha}$ if $m=1$.

By  Lemma \ref{FPS,3} and the fact that $n_2=0$, one gets $\deg[a_{N-1}y^{N-1}]\leq n_1+1$, which gives that $N\leq 1$. As the  system \eqref{Set1} is not solvable, one has $N=1$. But then, from \eqref{Set2}, for $n=0$, we have that $a_0$ must solve
$$ \dsty P_{m+2}(x)-xL\left[\frac{a_{0}(x)}{x^{n_1}}\right]= \left(\frac{r}{k}\Pi_{m+2}'(x) + \gamma\beta - (\alpha c - \gamma)x^m\right)x. $$
Substitution of  $a_{N-1}$ (corresponding to $N=1$) into this  expression yields 
\begin{multline*} \dsty  -\frac{\lambda_1r}{k}(x-k)(x^m+\beta)-\lambda_2(\gamma\beta -(\alpha c-\gamma)x^{m})-\frac{r}{k}\Pi_{m+2}(x) \frac{(\alpha  c-\gamma)\mbox{const.}}{\alpha} =\\ \frac{r}{k}\Pi_{m+2}'(x)  +\left(\left(1-\mbox{const.}\frac{\alpha c-\gamma}{\alpha}\right)x-\mbox{const.}_1\right)(\gamma\beta- (\alpha c - \gamma)x^m), 
\end{multline*}
which is not possible in view of  $(\alpha,r,c,k,\gamma,\beta)\in A\setminus B$.
\end{itemize}

  Therefore, systems \eqref{Set2} and \eqref{Set3} do not have polynomial solutions for $N>n_2$.

On account of \eqref{F1}, we have now the unique possibility 
$$N-n_2=0, \quad N\neq 0. $$
  From \eqref{Set4}, for $n=N+1$, we deduce that  $\lambda_1=-m$ and  $a_N(x)=\mbox{const.}x^{n_1}$. For   $n=N$, we see that  $a_{N-1}$  is the solution of 
$$ \dsty P_{m+2}(x)-\alpha x^{m+1} \left(\frac{a_{N-1}(x)}{x^{n_1}}\right)^{\prime}-xL\left[\mbox{const.}\right]= \left(\frac{r}{k}\Pi_{m+2}'(x) + \gamma\beta - (\alpha c - \gamma)x^m\right)x. $$
Noting  that the term $L\left[\mbox{const.}\right]$ is null, we obtain 
$$ \alpha \left(\frac{a_{N-1}(x)}{x^{n_1}}\right)^{\prime}= \left( r \beta (1 - m) - \beta \gamma (\lambda_2 + 1) \right) x^{-m}
+ \frac{(m - 2) r \beta}{k} x^{ - m+1} 
+ \left( r + c \alpha - \gamma +c \alpha \lambda_2 - \gamma \lambda_2 \right) -\frac{2r}{k} x.
$$
In such a case, we  have that $a_{N-1}$ can be expressed as 
\begin{eqnarray}\label{an-1}
a_{N-1}(x)=-\frac{r}{\alpha k} x^{2+n_1}+\mbox{lower order terms}.  
\end{eqnarray}
Assuming that $N=1$ and that $a_{N-1}$ is a polynomial,  we have from \eqref{an-1} that $a_0$ must satisfy 
$$ \dsty L\left[\frac{a_{0}(x)}{x^{n_1}}\right]=0. $$
This is not possible because in   $\dsty L\left[\frac{a_{0}(x)}{x^{n_1}}\right](0)$ the coefficient of  $x^{m+3}$  is $\dsty -\frac{2}{\alpha }\left(\frac{r}{k}\right)^2$, and by hypothesis $r\neq 0$.

On the other hand, if  $N\geq 2$,   one has for $n=N-1$
$$ \dsty  -\alpha x^{m}\left(\frac{a_{N-2}(x)}{x^{n_1}}\right)^{\prime}-L\left[\frac{a_{N-1}(x)}{x^{n_1}}\right]=0. $$
Assuming that this equation has  a polynomial solution, a straightforward calculation shows that 
$$\deg[a_{N-2}]=4+n_1. $$
This violates the inequality of Lemma \ref{FPS,3}.

Therefore, the set of differential equations \eqref{Set1}, \eqref{Set2}, \eqref{Set3}, and \eqref{Set4}  is not solvable in polynomials. So,  there  are no Darbouxian integrating factors for $L_{v_i}$.

  The special case $(\alpha,r,c,k,\gamma,\beta)\notin A$ reduces trivially to  items a) and b),  whereas for  $(\alpha,r,c,k,\gamma,\beta)\in B$,   $L_{v_i}$ only depends on the derivative of variable $x$, which is  item c). This completes the  proof of the theorem.  \lqqd

  \proof (\emph{Of Corollary \ref{Abel}})

 Let $U\subset\CC\times\CC$ be an open subset and  $F(x,y), (x,y)\in U$  a  Liouvillian first integral of the differential equation \eqref{AbDE}.  Then 
$$ \frac{\partial F}{\partial x} + \frac{\partial F}{\partial y}\frac{dy}{dx}=0, $$
and hence $L_{v_i}[F]=0, i=1,2$ in $U$. This implies  that $F$ is a Liouvillian first integral of the system \eqref{Ga1}.  By  Theorem \ref{Liov}  is not Liouvillian integrable, which completes the proof. \lqqd

 \section{Analytic first integrals} \label{S4}

\begin{lemma}\label{LF}
Let \(t\in\mathbb{C}\). Fix any branch of the square root and set
\[
s:=\sqrt{1-t},
\]
assumed so that the numerator or denominator below are nonzero, i.e. $s\neq \pm1$. Define
\[
F(t)=\frac{-1+s}{-1-s}=\frac{1-s}{1+s}.
\]
Then
\[
F(t)\notin\mathbb{Q}_{<0}
\quad\Longleftrightarrow\quad
t\notin\mathbb{R}\ \text{ or }\ t>0\ \text{ or }\ s\notin\mathbb{Q}.
\]

\end{lemma}
\proof
It suffices to prove, 
\[
F(t)\in\mathbb{Q}_{<0}
\quad\Longleftrightarrow\quad
t\in\mathbb{R},\ t<0,\ \text{and }s\in\mathbb{Q}.
\]
Write \(s=\sqrt{1-t}\), so that
\[
F(t)=\frac{1-s}{1+s}.
\]

Suppose first that \(F(t)\in\mathbb{Q}_{<0}\). We  have $\dsty s=\frac{1-F(t)}{1+F(t)}$ which implies  \(s\in\mathbb{Q}\).   Moreover, \(F(t)<0\) implies \(|s|>1\), so that
\[
t=1-s^{2}<0,\qquad t\in\mathbb{R}.
\]
Thus \(t\) is real negative and \(s\) is rational.

Conversely, if \(t\in\mathbb{R}\), \(t<0\), and \(s\in\mathbb{Q}\), then \(|s|>1\). Therefore
\[
F(t)=\frac{1-s}{1+s}\in\mathbb{Q},\qquad F(t)<0,
\]
so \(F(t)\in\mathbb{Q}_{<0}\).

This proves
\[
F(t)\in\mathbb{Q}_{<0}\;\;\Longleftrightarrow\;\;t\in\mathbb{R},\;t<0,\;\sqrt{1-t}\in\mathbb{Q},
\]
and negating gives the stated condition for \(F(t)\notin\mathbb{Q}_{<0}\). \lqqd

\proof \emph{(Of Theorem \ref{AI})}

Solving $L_{v_1}=0$ and $L_{v_2}=0$ yields the equilibrium points listed in \eqref{frakp} and \eqref{frakq}.

The Jacobian matrices are
\[
J_{v_1}(x,y)=\begin{pmatrix}
 -\frac{r}{k}\Pi'_{m+2}(x)-m\alpha x^{m-1}y & -\alpha x^m \\
 m(\alpha c-\gamma)x^{m-1}y & -\gamma \beta+(\alpha c-\gamma)x^m
\end{pmatrix},
\]
and
\[
J_{v_2}(x,y)=\begin{pmatrix}
 -\frac{r}{k}(2x-k) & -\alpha \\
 0 & -\gamma+\alpha c
\end{pmatrix}.
\]

For a \(2\times2\) the  characteristic polynomial may be written as \(\lambda^2 +B\lambda + C\), where \(B=-\operatorname{tr}(J)\) and \(C=\det(J)\). 

Set \(\lambda_{1,2}=(-B\pm\sqrt{B^2-4C})/2\). The condition $C\neq 0$ guarantees that at least one   eigenvalue is  nonzero. Moreover, we may  assume $B\neq 0$. Indeed, if  $B=0$, then  $\dsty\frac{\lambda_1}{\lambda_2}=-1$. Hence,  if $B\neq 0$, one has
\[
\frac{\lambda_1}{\lambda_2}=\frac{-1+\sqrt{1-\frac{4C}{B^2}}}{-1-\sqrt{1-\frac{4C}{B^2}}}, 
\]
where we fix  any branch of the root.  It follows from Lemma \ref{LF}  that the nonresonance requirement   \(\dsty \frac{\lambda_1}{\lambda_2}\notin\mathbb Q_{<0}, \lambda_1,\lambda_2\neq 0\)   for the Jacobian matrix $J$ at an equilibrium point $p$  is equivalent to 
$$\frac{4D}{T^2}\notin\mathbb{R}\ \text{ or }\ \frac{4D}{T^2} >0\ \text{ or }\ \sqrt{1-\frac{4D}{T^2}}\notin\mathbb{Q}\   \text{ and  }\  T,D\neq 0, $$
where \(T=-\operatorname{tr}(J)\) and \(D=\det(J)\).

Below we compute  $T$ and  $D$ at each equilibrium and apply the previous criterion.

\begin{itemize}
\item At \(\mathfrak p_{1,j}\), put \(x_*= \sqrt[m]{\beta\gamma/(\alpha c-\gamma)}\)). The Jacobian equals
\[
J(\mathfrak p_{1,j})=\begin{pmatrix}
 -\frac{r}{k}\Pi'_{m+2}(x_*)+ \dfrac{mr}{k x_*}\Pi_{m+2}(x_*)& -\dfrac{\alpha\beta\gamma}{\alpha c-\gamma}\\[6pt]
 -\dfrac{m(\alpha c-\gamma)r\Pi_{m+2}(x_*)}{kx_*\alpha}& 0
\end{pmatrix}.
\]
Thus
\[
T=-\frac{r}{k}\Pi'_{m+2}(x_*)+ \frac{mr\alpha\beta c(x_*-k)}{k(\alpha c-\gamma)},\qquad
D=-\frac{m\alpha\beta^2\gamma r c(x_*-k)}{k(\alpha c-\gamma)}.
\]
Applying the condition above yields the first row of Table \ref{betan0}. 

\item At \(\mathfrak p_2=(k,0)\) the Jacobian is
\[
J(\mathfrak p_2)=\begin{pmatrix}
 -\frac{r}{k}\Pi'_{m+2}(k)& -\alpha k^m \\[4pt]
 0 & -\gamma\beta+(\alpha c-\gamma)k^m
\end{pmatrix},
\]
so
\[
T=-\frac{r}{k}\Pi'_{m+2}(k)-\gamma\beta+(\alpha c-\gamma)k^m,\qquad
D=\frac{r}{k}\Pi'_{m+2}(k)\bigl(\gamma\beta-(\alpha c-\gamma)k^m\bigr),
\]
which yields the second row of Table \ref{betan0} via the same test.

\item At \(\mathfrak p_3=(0,0)\),
\[
J(\mathfrak p_3)=\begin{pmatrix} r\beta & 0\\ 0 & -\gamma\beta\end{pmatrix},
\]
hence \(T=r\beta-\gamma\beta\), \(D=-r\gamma\beta^2\), which gives the third row.

\item At \(\mathfrak p_{4,j}=(\sqrt[m]{-\beta},0)\) (write \(x_*=\sqrt[m]{-\beta}\)),
\[
J(\mathfrak p_{4,j})=\begin{pmatrix}
 -\frac{r}{k}\Pi'_{m+2}(x_*)&\alpha\beta\\[4pt]
 0 & -\alpha\beta c
\end{pmatrix},
\]
so
\[
T=-\frac{r}{k}\Pi'_{m+2}(x_*)-\alpha\beta c,\qquad
D=\frac{r}{k}\Pi'_{m+2}(x_*)\alpha\beta c,
\]
giving the fourth row.

\end{itemize}

Finally, we perform the analogous explicit computations for \(L_{v_2}\) in the \(\beta=0\) case, 
$$J(\mathfrak{q}_1)= \left(
\begin{array}{cc}
 -r & -\alpha \\
0 & -\gamma + \alpha c \\
\end{array}
\right),  J(\mathfrak{q}_2)= \left(
\begin{array}{cc}
 r & -\alpha \\
0 & -\gamma + \alpha c \\
\end{array}
\right) $$
 to obtain Table \ref{bet0}.

This completes the proof.
\lqqd

 \proof (\emph{Of Theorem \ref{SFunct}})
 
For $\gamma=0$, from  \eqref{Ga1} we obtain  the Riccati differential equation 
$$
  \frac{dx}{dy}=-\frac{r}{\alpha cky}x^2+\frac{r}{\alpha cy}\left(1-\frac{\beta}{k}\right)x+\frac{\beta r}{\alpha cy}-\frac{1}{c}.
$$ 
If   $r,\alpha,c,k \in \CC\setminus\{0\}$,  by \cite[p.103]{Hi97}, the map 
\begin{eqnarray}\label{map1}
\dsty x(y)\mapsto \frac{\alpha cky}{r}\frac{w^{\prime}(y)}{w(y)}
\end{eqnarray}
 transforms  the Riccati equation into the second order linear equation 
\begin{eqnarray}\label{Tric}
w^{\prime\prime}+\frac{1}{y}\left(1-\frac{r}{\alpha c}+\frac{r\beta}{\alpha ck}\right)w^{\prime}+\left(\frac{r}{\alpha c^2 ky}-\frac{\beta r^2}{\alpha^2 c^2ky^2}\right)w=0.
\end{eqnarray}
By relations  \cite[(13),(16) p.251]{MaBa53}, the general solution  of \eqref{Tric} is given by
$$w(y)=y^{\frac{r(k-\beta)}{2\alpha ck}}\left(c_1J_{\nu}\left(  \frac{2}{c}\sqrt{\frac{ry}{\alpha k}}  \right)+c_2Y_{\nu}\left(  \frac{2}{c}\sqrt{\frac{ry}{\alpha k}}  \right)  \right). $$

 In view of \eqref{map1},  the Riccati equation admits  the explicit solution 
$$x= \frac{\alpha cky}{r}\frac{w^{\prime}(y)}{w(y)},$$
which  leads to the relation 
$$\frac{c_1\frac{1}{c}\sqrt{\frac{r}{\alpha k}}y^{-\frac{1}{2}}   J'_{\nu}\left(  \frac{2}{c}\sqrt{\frac{ry}{\alpha k}}  \right)+c_2\frac{1}{c}\sqrt{\frac{r}{\alpha k}}y^{-\frac{1}{2}}  Y'_{\nu}\left(  \frac{2}{c}\sqrt{\frac{ry}{\alpha k}} \right)}{c_1J_{\nu}\left(  \frac{2}{c}\sqrt{\frac{ry}{\alpha k}}  \right)+c_2Y_{\nu}\left(  \frac{2}{c}\sqrt{\frac{ry}{\alpha k}}\right) }=\frac{r}{y\alpha c k}\left(x-\frac{k-\beta}{2}\right). $$
For   $y\neq 0$,  the above  identity is equivalent to
\begin{multline*}
 c_1\left(\frac{1}{c}\sqrt{\frac{r}{\alpha k}}y^{\frac{1}{2}}   J'_{\nu}\left(  \frac{2}{c}\sqrt{\frac{ry}{\alpha k}}  \right)-    \frac{r}{\alpha c k}\left(x-\frac{k-\beta}{2}\right)J_{\nu}\left(  \frac{2}{c}\sqrt{\frac{ry}{\alpha k}}  \right)\right)+\\
 c_2\left( \frac{1}{c}\sqrt{\frac{r}{\alpha k}}y^{\frac{1}{2}}   Y'_{\nu}\left(  \frac{2}{c}\sqrt{\frac{ry}{\alpha k}}  \right)-    \frac{r}{\alpha c k}\left(x-\frac{k-\beta}{2}\right)Y_{\nu}\left(  \frac{2}{c}\sqrt{\frac{ry}{\alpha k}}  \right)  \right )=0. \end{multline*}
In  the shorthand notation, 
$$c_1A(x,y)+c_2B(x,y)=0.$$

Assume that $U\subset\mathbb{C}^2$ is open and that $A$ and $B$ are not simultaneously zero on $U$.
Then, for any constants $(c_1,c_2)\neq(0,0)$, the relation
\[
c_1A(x,y)+c_2B(x,y)=0
\]
is invariant along the trajectories of system~\eqref{Ga1}.
In particular, on any connected component of $U$ where $B\neq0$, the function
\[
H(x,y)=\frac{A(x,y)}{B(x,y)}
\]
is constant along solutions of~\eqref{Ga1}, and therefore defines a first integral in $U$.
The cases $c_1=0$ or $c_2=0$ correspond to the invariant curves $B(x,y)=0$ and $A(x,y)=0$, respectively.
\lqqd


 \section*{Acknowledgements}

The author would like to thank the reviewers for their careful reading of the manuscript and for their constructive comments and suggestions, which helped improve the clarity and presentation of the paper.

\bibliography{biblio}

\end{document}